\documentclass[reqno]{amsart}

\usepackage{color}

\newtheorem{theorem}{Theorem}[section]
\newtheorem{lemma}[theorem]{Lemma}

\newtheorem{corollary}[theorem]{Corollary}
\theoremstyle{definition}

\theoremstyle{remark}

\numberwithin{equation}{section}

\def\bfR{\mbox{\boldmath$R$}}
\begin{document}

\title{
Global dynamics of the chemostat with variable yields}
\author{Tewfik Sari}
\address{Laboratoire de Math\'ematiques, Informatique et Applications, 
Universit\'e de Haute Alsace, 4 rue des fr\`eres Lumi\`ere, 68093 Mulhouse, and 
EPI MERE INRIA-INRA, UMR MISTEA, 
2, pl. Viala, 34060 Montpellier, France}
\email{Tewfik.Sari@uha.fr}

\date{\today}

\begin{abstract}
In this paper, we consider a competition model between $n$ species in a chemostat including 
both monotone and non-monotone response functions, distinct removal rates and variable yields.
We show that only the species with the lowest break-even concentration survives, provided that additional technical 
conditions on the growth functions and yields are satisfied. LaSalle's extension theorem of the Lyapunov stability theory
is the main tool.
\end{abstract}



\pagestyle{myheadings}
\thispagestyle{plain}
\markboth{T. SARI AND G.S.K. WOLKOWICZ}{COMPETITION IN THE CHEMOSTAT}

\maketitle

\section{Introduction}
In this paper we study the global dynamics of the following model of the chemostat in which $n$ populations of microorganisms compete for a single growth-limiting substrate:
\begin{equation}\label{eqsxi}
\begin{array}{lcll} 
 \displaystyle S'(t)& =&D[S^0-S(t)]-
\sum_{i=1}^nf_i(S(t))x_i(t)&\\[2mm]
 \displaystyle x'_i(t)&=& [p_i(S(t)) - D_i]x_i(t),& i=1\cdots n,
\end{array}
\end{equation}
where $S(0)\geq 0$ and $x_i(0)>0$, $i=1\cdots n$ and $S^0$, $D$ and $D_i$ are positive constants.
In these equations, $S(t)$ denotes the concentration of the substrate at time $t$; $x_i(t)$ denotes the concentration of the $i$th population of microorganisms at time $t$; $f_i(S)$ represents the uptake rate of substrate of the $i$th population;
$p_i(S)$ represents the per-capita growth rate of the $i$th population and so 
the function $y_i(S)$ defined by 
$$y_i(S)=\frac{p_i(S)}{f_i(S)}$$
is the growth yield; 
$S^0$ and $D$ denote, respectively, the concentration of substrate in the feed bottle and the flow rate of the chemostat; each 
$D_i$ represents the removal rate of the $i$th population. 
We make the following assumptions on the functions $p_i$ and $f_i$:
\begin{itemize}
 \item 
$p_i,f_i:\bfR_+\to\bfR_+$ are continuous,
\item
$p_i(0)=f_i(0)=0$  and for all $S>0$, $p_i(S)>0$ and $f_i(S)>0$,
\end{itemize}
For general background on model (\ref{eqsxi}), in the constant yield case $y_i(S)=Y_i$, 
the reader is referred to the monograph of Smith and Waltman \cite{chem}.
Following Butler and Wolkowicz \cite{bw}, 
we make the following assumptions on the form of the response functions $p_i$:
there exist positive extended real numbers $\lambda_i$ and $\mu_i$ with 
$\lambda_i\leq \mu_i\leq +\infty$ such that
$$
\begin{array}{lcl}
p_i(S)<D_i&\mbox{ if }&S\notin[\lambda_i,\mu_i]\\
p_i(S)>D_i&\mbox{ if }&S\in]\lambda_i,\mu_i[.
\end{array}
$$
Hence there are at most two values of $S$, $S=\lambda_i$ and $S=\mu_i$, called the break-even concentrations,
satisfying the equation $p_i(S)=D_i$. 
We adopt the convention $\mu_i= \infty$ if this equation has only one solution and $\lambda_i= \infty$ if 
it has no solution.

The global analysis of this model was considered by Hsu, Hubbell and Waltman \cite{hhw}, 
in the Monod case \cite{monod} when the response functions are of Michaelis-Menten form, 
\begin{equation}\label{MichaelisMenten}
p_i(S)=\frac{a_iS}{b_i+S},
\end{equation}
and the yields are constant $y_i(S)=Y_i$, and $D_i=D$ for $i=1\cdots n$.
The authors showed that only the species with the lowest break-even concentration survives. 
Thus the competitive exclusion principle (CEP) holds: only one species survives, namely 
the species which makes optimal use of the resources.
Hsu \cite{hsu} applied a Lyapunov-LaSalle argument 
to give a simple and elegant proof of the result in \cite{hhw} for the case of different removal rates $D_i$. 
The Lyapunov function $V_H$ used by Hsu is
\begin{equation}\label{HSU}
V_{H}=
\int_{\lambda_1}^S\frac{\sigma-\lambda_1}{\sigma}d\sigma+
c_1\int_{x_{1}^*}^{x_1}\frac{\xi-x_{1}^*}{\xi}d\xi+
\sum_{i=2}^nc_ix_i,
\end{equation}
where 
$$\quad c_i=\frac{1}{Y_i}\frac{a_i}{a_i-D_i},
\quad i=1\cdots n,\quad\mbox{ and }\quad x_{1}^*=DY_1\frac{S^0-\lambda_1}{D_1}.$$

CEP has been proved under a variety of hypothesis \cite{amg,bw,li,wl,wx}. For a survey 
on the contribution of each paper the reader may consult the introduction of the papers \cite{LLS,li,SM}.
The variable yield case was considered, for $n=1$ and $n=2$ by Pilyugin and Waltman \cite{PW}, with a particular interest to linear and quadratic yields, and by 
Huang, Zhu and Chang \cite{HZC}.
The model (\ref{eqsxi}) was considered by Arino, Pilyugin and Wolkowicz \cite{apw} and  Sari and Mazenc \cite{SM}. 

In this paper we generalize \cite{wl} by allowing variable yields and we generalize \cite{apw} by allowing multi species. We further extend \cite{SM} by providing less restrictive assumptions on the system.
For biological motivations concerning the dependance of the yields on the substrate, the reader is refeered to \cite{apw,PW} and the references therein.

\section{Analysis of the model}
It is known (see Theorem 4.1 \cite{apw}) that the non-negative cone is invariant under the flow of (\ref{eqsxi}) 
and all solutions are defined and remain bounded for all $t\geq 0$.
System (\ref{eqsxi}) can have many equilibria: the washout equilibrium
$E_0=(S^0,0,\cdots,0)$, which is locally exponentially stable if and only if for all $i=1\cdots n$, $S^0\notin[\lambda_i,\mu_i]$
and the equilibria 
$E_{i}^*$ and $E_{i}^{**}$
where all component of $E_{i}^*$ and $E_{i}^{**}$ vanish except for the first and the $(i+1)$th, which are
$$S=\lambda_i,\qquad x_i=x_{i}^*:=F_i(\lambda_i),\quad\mbox{ for }E_{i}^*$$
and
$$S=\mu_i,\qquad x_i=x_{i}^{**}:=F_i(\mu_i),\quad\mbox{ for }E_{i}^{**}$$
respectively, where
$$
F_i(S)=D\displaystyle\frac{S^0-S}{f_i(S)}.
$$
The equilibrium $E_{i}^*$ lies in the non-negative cone if and only 
if $\lambda_i\leq S^0$. If  
$\lambda_i<\lambda_j$ for all $i\neq j$ and $F_i'(\lambda_i)<0$ then it is locally exponentially stable. 
It coalesces with $E_0$ when $\lambda_i=S^0$. 
The equilibrium $E_{i}^{**}$ lies in the non-negative cone if and only if $\mu_i\leq S^0$ 
and is locally exponentially unstable if it exists.
Its coalesces with $E_0$ when $\mu_i=S^0$. 
Besides these equilibria, the system (\ref{eqsxi}) can have a continuous
set of non-isolated equilibria in the non-generic cases where two or more of the break-even concentrations
are equal. 
In what follows we assume, that 
\begin{equation}\label{condition}
\lambda_1<\lambda_2\leq\cdots\leq\lambda_n,\qquad\mbox{and}\qquad\lambda_1<S^0<\mu_1.
\end{equation}
Hence $E_0$ is locally exponentially unstable and the equilibrium $E_{1}^*$ lies in 
the non-negative cone. It is 
locally exponentially stable if and only if
\begin{equation}\label{locexpstab}
F_1'(\lambda_1)<0\Longleftrightarrow f_1(\lambda_1)+f_1'(\lambda_1)(S^0-\lambda_1)>0.
\end{equation}
We consider the global asymptotic stability of $E_{1}^*$.

Before presenting the results, we need the following lemma,
\begin{lemma}\label{SppS0}
The solutions $S(t)$, $x_i(t)$, $i=1\cdots n$ of (\ref{eqsxi}) with positive initial conditions are 
positive and bounded, and if $\lambda_i<S^0<\mu_i$ for some $i=1\cdots n$, then $S(t)<S^0$ for all sufficiently large $t$.
\end{lemma}
\begin{proof}
The proof is similar
to the proof of Lemma 2.1 in \cite{wl} obtained for the model (\ref{eqsxi}) in the case where the yields are constant.
\end{proof}

\begin{theorem}\label{ourthm}
Assume that 
\begin{enumerate}
\item 
$\lambda_1<\lambda_2\leq\cdots\leq\lambda_n$, and $\lambda_1<S^0<\mu_1$.
\item 
There exist constants
$\alpha_i>0$ for each $i\geq 2$ satisfying $\lambda_i<S^0$ such that
\begin{equation}\label{conditionci}
\max_{0<S<\lambda_1}g_i(S)\leq \alpha_i\leq \min_{\lambda_i<S<\rho_i}g_i(S),\qquad i\geq 2.
\end{equation}
where 
\begin{equation}\label{gi}
g_i(S)=\frac{f_i(S)}{f_1(\lambda_1)}\frac{p_1(S)-D_1}{p_i(S)-D_i}\frac{S^0-\lambda_1}{S^0-S}.
\end{equation}
\item 
The function 
\begin{equation}\label{F}
F(S)=\displaystyle\frac{f_1(S)}{S^0-S}.
\end{equation} 
satisfies  
\begin{equation}\label{conditioni}
 F(S)<F(\lambda_1)\mbox{ if }S\in]0,\lambda_1[,
\qquad
 F(S)>F(\lambda_1)\mbox{ if }S\in]\lambda_1,S^0[.
\end{equation} 
\end{enumerate}
Then the equilibrium $E_{1}^*$ is globally asymptotically 
stable for system (\ref{eqsxi}) with respect to the interior of the positive cone.
\end{theorem}
\begin{proof}
From Lemma \ref{SppS0} it follows that there is no loss of generality in restricting our attention to $0\leq S<S^0$.
Consider the function $V=V(S,x_1,\cdots,x_n)$ as follows
\begin{equation}\label{lyapunov}
V=
\int_{\lambda_1}^S\frac{(p_1(\sigma)-D_1)(S^0-\lambda_1)}{f_1(\lambda_1)(S^0-\sigma)}d\sigma+
\int_{x_{1}^*}^{x_1}\frac{\xi-x_{1}^*}{\xi}d\xi+
\sum_{i=2}^n\alpha_ix_i
\end{equation}
where $\alpha_i$ are the positive constants satisfying (\ref{conditionci}). The function $V$ is continuously 
differentiable for $0<S<S^0$ and $x_i>0$ and positive except at the point $E_{1}^*$. 
The derivative of $V$ along the trajectories of (\ref{eqsxi}) is given by
$$
V'=x_1[p_1(S)-D_1]\frac{F(\lambda_1)-F(S)}{F(\lambda_1)}+\sum_{i=2}^nx_ih_i(S)
$$
where 
$$h_i(S)=[p_i(S)-D_i][\alpha_i-g_i(S)]$$ 
and $g_i(S)$ and $F(S)$ are given by (\ref{gi}) and (\ref{F}) respectively.
First, note that, using (\ref{condition}) and (\ref{conditioni}), the first term of the above sum is always 
non-positive
for $0<S<S^0$ and equals 0 for $S\in]0,S^0[$ if and only if $S=\lambda_1$ or $x_1=0$. 
If $S\in[\lambda_1,\lambda_i]$ then $p_i(S)<D_i$ and $p_1(S)>D_1$ so that $g_i(S)<0<\alpha_i$ for any choice of $\alpha_i>0$.
Similarly if $\mu_i<S^0$ and 
$S\in[\mu_i,S^0]$ then $p_i(S)<D_i$ and $p_1(S)>D_1$ so that $g_i(S)<0<\alpha_i$ for any choice of $\alpha_i>0$.
On the other hand, if 
$S\in[0,\lambda_1]$ then $p_i(S)<D_i$ and, using (\ref{conditionci}), $g_i(S)\leq\alpha_i$ so that $h_i(S)<0$. 
Finally, if $S\in[\lambda_i,\rho_i]$ then $p_i(S)>D_i$ and $g_i(S)\geq\alpha_i$ so that $h_i(S)<0$.
Thus $h_i(S)<0$ for every $S\in]0,S^0[$, provided that the numbers $\alpha_i$ satisfy (\ref{conditionci}).
Hence $V'\leq 0$ and
$V'=0$ if and only if $x_i=0$ for $i=1\cdots n$ or 
$S=\lambda_1$ and $x_i=0$ for $i=2\cdots n$.
We use LaSalle theorem 
(the details are as in the proof of Theorem 2.3 in \cite{wl}), 
the $\omega$-limit set of the trajectory is $E_1^*$.
\end{proof}

Condition (\ref{conditioni}) is satisfied in the particular case where  
$F'(S)>0$ for $0<S<S^0$.  
For $S=\lambda_1$ we obtain the condition
(\ref{locexpstab}) of local exponential stability of $E_{1}^*$.

\section{Applications}
In this section we show how Theorem \ref{ourthm} extends results in \cite{apw,SM,wl}.

\subsection{One species}
In the case $n=1$ the equations take the form
\begin{equation}\label{eqsx}
\begin{array}{l} 
S' = D(S^0-S)-x_1f_1(S)\\
x_1'= [p_1(S) - D_1]x_1
\end{array}
\end{equation} 
\begin{corollary}\label{thmapw}
If 
$\lambda_1<S^0<\mu_1$ and 
$
1-\frac{f_1(S)(S^0-\lambda_1)}{f_1(\lambda_1)(S^0-S)}
$
has exactly one sign change for $S\in(0,S^0)$ then $E_1^*$ is globally asymptotically stable
with respect to the interior of the positive quadrant.
\end{corollary}
\begin{proof}
The result follows from Theorem \ref{ourthm} since the hypothesis on the change of sign is  
is equivalent to (\ref{conditioni}). 
In the case
when $n=1$ the condition (\ref{conditionci}) is obviously satisfied.
\end{proof}
Corollary \ref{thmapw} was obtained by Arino, 
Pilyugin and Wolkowicz (see \cite{apw}, Theorem 2.11). These 
authors used the following Lyapunov function
$$V_{APW}=\int_{\lambda_1}^S\frac{(p_1(\sigma)-D_1)(S^0-\lambda_1)}{f_1(\lambda_1)(S^0-\sigma)}d\sigma+
\int_{x^*_1}^{x_1}\frac{\xi-x^*_1}{\xi}d\xi.$$
Notice that the Lyapunov function (\ref{lyapunov}) we use reduces to the function $V_{APW}$ in the case when $n=1$.

\subsection{Constant yields}
In the case when the yields are constant, $y_i(S)=Y_i$, the equations take the form
\begin{equation}\label{constyields}
\begin{array}{lcll} 
 \displaystyle S'& =&D[S^0-S]-
\sum_{i=1}^n\frac{p_i(S)}{Y_i}x_i&\\[2mm]
 \displaystyle x'_i&=& [p_i(S) - D_i]x_i,& i=1\cdots n,
\end{array}
\end{equation}
\begin{corollary}\label{thmwl}
Assume that 
\begin{enumerate}
\item 
$\lambda_1<\lambda_2\leq\cdots\leq\lambda_n$, and $\lambda_1<S^0<\mu_1$.
\item 
There exist constants
$\alpha_i>0$ for each $i\geq 2$ satisfying $\lambda_i<S^0$ such that
\begin{equation}\label{conditionciWL}
\max_{0<S<\lambda_1}g_i^{WL}(S)\leq \alpha_i\leq \min_{\lambda_i<S<\rho_i}g_i^{WL}(S),\qquad i\geq 2.
\end{equation}
where 
\begin{equation}\label{giWL}
g_i^{WL}(S)=\frac{p_i(S)}{D_1}\frac{p_1(S)-D_1}{p_i(S)-D_i}\frac{S^0-\lambda_1}{S^0-S}.
\end{equation}
\end{enumerate}
Then the equilibrium $E_{1}^*$ is globally asymptotically 
stable for system (\ref{constyields}) with respect to the interior of the positive cone.
\end{corollary}
\begin{proof}
In the case when the yields are constant, we have 
$$g_i(S)=\frac{Y_1}{Y_i}g_i^{WL}(S).$$
Hence the conditions (\ref{conditionciWL}) imply the conditions (\ref{conditionci}) with constants $\alpha_i$ replaced by $\alpha_iY_1/Y_i$. 
On the other hand  $F(S)=\frac{p_1(S)}{(S^0-S)Y_1}$. Thus, assumption (\ref{conditioni}) follows from hypothesis $\lambda_1<S^0<\mu_1$. 
The result follows from Theorem \ref{ourthm}.
\end{proof}
Corollary \ref{thmwl} was obtained by Wolkowicz and Lu (see \cite{wl}, Theorem 2.3). These 
authors used the following Lyapunov function
$$
V_{WL}=
\int_{\lambda_1}^S\frac{(p_1(\sigma)-D_1)(S^0-\lambda_1)}{D_1(S^0-\sigma)}d\sigma+
\frac{1}{Y_1}\int_{x_{1}^*}^{x_1}\frac{\xi-x_{1}^*}{\xi}d\xi+
\sum_{i=2}^n\frac{\alpha_i}{Y_i}x_i.
$$
Notice that the Lyapunov function (\ref{lyapunov}) we use satisfies $V=Y_1V_{WL}$
in the case when the yields are constant. 

\subsection{Variable yields}
In \cite{SM}, another Lyapunov function is proposed in the case when the yields are variable,
leading to the following result.
\begin{corollary}\label{thmSM}
Assume that 
\begin{enumerate}
\item 
$\lambda_1<\lambda_2\leq\cdots\leq\lambda_n$, and $\lambda_1<S^0<\mu_1$.
\item 
There exist constants
$\alpha_i>0$ for each $i\geq 2$ such that
\begin{equation}\label{conditionciSM}
\max_{0<S<\lambda_1}g_i^{SM}(S)\leq \alpha_i\leq \min_{\lambda_i<S<\rho_i}g_i^{SM}(S),\qquad i\geq 2.
\end{equation}
where 
\begin{equation}\label{giSM}
g_i^{SM}(S)=\frac{f_i(S)}{f_1(S)}\frac{p_1(S)-D_1}{p_i(S)-D_i}.
\end{equation}
\item 
The function 
$
F(S)=\displaystyle\frac{f_1(S)}{S^0-S}
$ 
satisfies  
\begin{equation}\label{conditioniSM}
 F(S)<F(\lambda_1)\mbox{ if }S\in]0,\lambda_1[,
\qquad
 F(S)>F(\lambda_1)\mbox{ if }S\in]\lambda_1,S^0[.
\end{equation} 
\end{enumerate}
Then the equilibrium $E_{1}^*$ is globally asymptotically 
stable for system (\ref{eqsxi}) with respect to the interior of the positive cone.
\end{corollary}
\begin{proof}
The condition (\ref{conditioniSM}) is the same as the condition (\ref{conditioni}).
On the other hand we have 
$$g_i(S)=\frac{S^0-\lambda_1}{f_1(\lambda_1)}F(S)g_i^{SM}(S).$$
By (\ref{conditionciSM}) and (\ref{conditioniSM}) we have (\ref{conditionci}) with appropriate constants $\alpha_i$. 
The result follows from Theorem \ref{ourthm}.
\end{proof}
Corollary \ref{thmSM} was obtained by Sari and Mazenc (see \cite{SM}, Theorem 2.2). These 
authors used the following Lyapunov function
$$
V_{SM}=
\int_{\lambda_1}^S\frac{p_1(\sigma)-D_1}{f_1(\sigma)}d\sigma+
\int_{x_{1}^*}^{x_1}\frac{\xi-x_{1}^*}{\xi}d\xi+
\sum_{i=2}^n\alpha_ix_i
$$
where $\alpha_i$, $i=2\cdots n$ are positive constants satisfying (\ref{conditionciSM}).
In the case when the response functions are of Michaelis-Menten form 
(\ref{MichaelisMenten})
and the yields are constant the Lyapunov function $V_{SM}$ reduces to the Lyapunov 
function $V_{H}$ (\ref{HSU}) useb by Hsu \cite{hsu}. 
Notice that the Lyapunov function (\ref{lyapunov}) we use is not proportional to the function $V_{SM}$.

\section{Discussion}
In this paper we considered a mathematical model (\ref{eqsxi}) 
of  $n$ species of microorganisms in competion in a chemostat
for a single resource. The model incorporates 
both monotone and non-monotone response functions, distinct removal rates and variable yields.
We demonstrated that the CEP holds for a large class of growth functions and yields.

Even with constant yields, the problem is not yet completely solved: the CEP holds for a large 
class of growth functions \cite{amg,bw,hsu,li,wl,wx} but
an important open question remains: {\em is the CEP
true assuming only that the $f_i$ are monotone with no restriction on the $D_i$ ?}  This 
major open problem remains unresolved after more than thirty years \cite{LLS}.
However, in the case of constant yields numerical simulations 
of model (\ref{eqsxi}) have only displayed 
competitive exclusion. 

In the case where the yields are constant, it is known \cite{bw} 
that the CEP holds provided that $D_i=D$ for all $i$, the set 
$Q=\bigcup_{i\in N}]\lambda_i,\mu_i[$ is connected, and $S^0\in Q$, where $N=\{i:\lambda_i<S^0\}$. 
Wolkowicz and Lu \cite{wl} conjectured that this result can be extented to the case of different removal rates. 
Under condition (\ref{condition}), it is clear that the set 
$Q$ is connected, and $S^0\in Q$. 
The condition $\lambda_1<\lambda_i$ for $i\neq 1$ can be stated without loss of generality, 
by labelling the populations
such that the index $i=1$ corresponds to the lowest break-even concentration,
but the condition $\lambda_1<S^0<\mu_1$ in (\ref{condition}) cannot be stated without loss of generality. 
If $\mu_1<S^0$, it is not possible to show the CEP by the methods that we used.
To the best of our knowledge, in the case of different removal rates and non-monotone response functions, the 
CEP has been proved only under the assumption $S^0<\mu_1$ \cite{li,wl,wx}. 
However, Rapaport and Harmand \cite{RH} considered the case of two populations and 
proposed conditions on the growth functions such that the CEP holds under the condition $\mu_1<S^0$. 
It should be interesting to extend their methods to
more general cases. We leave this problem for future investigations.

Our result concern also
the case of variable yields, for which it is known \cite{apw,HZC,PW} that more exotic dynamical behaviours, including limit cycles, are possible. Thus in the case of variable yields, it is of great importance to have criteria ensuring the global convergence to an equilibrium with at most one surviving species.

\end{document}